\documentclass[12pt]{amsart}
\usepackage{amscd, amsfonts, amssymb}
\usepackage{fullpage}
\usepackage{latexsym}
\usepackage{verbatim}
\usepackage{enumerate}
\usepackage[colorlinks=true]{hyperref}

\newcommand\inverse{{^{-1}}}

\newcommand\ra{\rightarrow}

\newcommand{\FF}{{\mathbb F}}

\newcommand{\tuple}[1]{{\mathbf {#1}}}
\DeclareMathOperator{\Int}{Int}

\DeclareMathOperator{\Aut}{Aut}
\DeclareMathOperator{\Out}{Out}

\DeclareMathOperator{\GL}{GL}

\DeclareMathOperator{\SL}{SL}

\DeclareMathOperator{\PGL}{PGL}

\DeclareMathOperator{\Mat}{Mat}

\numberwithin{equation}{section}

\newtheorem{thm}[equation]{Theorem}
\newtheorem*{Thm}{Theorem}
\newtheorem{lem}[equation]{Lemma}
\newtheorem{cor}[equation]{Corollary}
\newtheorem{prop}[equation]{Proposition}

\newtheorem*{Alg}{Algorithm}

\theoremstyle{definition}
\newtheorem{defn}[equation]{Definition}
\newtheorem{exmp}[equation]{Example}

\theoremstyle{remark}
\newtheorem{rem}[equation]{Remark}
\theoremstyle{remark}


\title[$G$-complete reducibility in non-connected groups]
{$G$-complete reducibility in non-connected groups}

\author[M.\  Bate]{Michael Bate}
\address
{Department of Mathematics,
University of York,
York YO10 5DD,
United Kingdom}
\email{michael.bate@york.ac.uk}

\author[S. Herpel]{Sebastian Herpel}
\address
{Fachbereich Mathematik,
Technische Universit\"at Kaiserslautern,
Postfach 3049,
67653 Kaiserslautern, Germany}
\email{herpel@mathematik.uni-kl.de}

\author[B.\ Martin]{Benjamin Martin}
\address
{Department of Mathematics,
University of Auckland,
Private Bag 92019,
Auckland 1142,
New Zealand}
\email{Ben.Martin@auckland.ac.nz}

\author[G. R\"ohrle]{Gerhard R\"ohrle}
\address
{Fakult\"at f\"ur Mathematik,
Ruhr-Universit\"at Bochum,
D-44780 Bochum, Germany}
\email{gerhard.roehrle@rub.de}

\begin{document}

\begin{abstract}
In this
paper we present an algorithm for determining whether a subgroup $H$ of
a non-connected reductive group $G$ is $G$-completely reducible.  The algorithm consists of
a series of reductions; at each step, we perform operations involving
connected groups, such as checking whether a certain subgroup of $G^0$
is $G^0$-cr.  This essentially reduces the problem of determining
$G$-complete reducibility to the connected case.
\end{abstract}

\maketitle


\section{Introduction}
\label{sec:intro}

Let $G$
be a connected reductive linear algebraic group over an algebraically
closed field of characteristic 0 or characteristic $p>0$.  
Following Serre \cite{serre1}, we say a
subgroup $H$ of $G$ is \emph{$G$-completely reducible} ($G$-cr) if
whenever $H$ is contained in a parabolic subgroup $P$ of $G$, then $H$ is
contained in some Levi subgroup of $P$.  The definition extends to
non-connected reductive $G$ as well: one replaces parabolic and Levi
subgroups with so-called Richardson parabolic and Richardson Levi
subgroups respectively (see \cite{BMR}, \cite{martin1}
and Section \ref{sec:prelims}).

Even if one is interested mainly in connected reductive groups, one must
sometimes consider non-connected groups.  For instance, natural subgroups
of a connected group, such as normalizers and centralizers, are often
non-connected.  The notion of 
$G$-complete reducibility is much better understood in
the connected case, e.g.,\ see \cite{BHMR}, \cite{liebeckseitz0},
and \cite{liebeckseitz}.
In this
paper we present an algorithm for determining whether a subgroup $H$ of
a non-connected reductive group $G$ is $G$-cr.  The algorithm consists of
a series of reductions; at each step, we perform operations involving
connected groups, such as checking whether a certain subgroup of $G^0$
is $G^0$-cr.  This essentially reduces the problem of determining
$G$-complete reducibility to the connected case.

An important special case of the general problem described above is the
following.  Let $H$ be a subgroup of $G$.  We say $H$\emph{acts on $G^0$
by outer automorphisms} if for each $1\neq h\in H$, conjugation by $h$
gives a non-inner automorphism of $G^0$.  In this case, we may identify
$H$ with a subgroup of $\Out(G^0)$.  Now suppose also that $G^0$ is
simple; then $H$ is cyclic except for possibly when $G^0$ is of type
$D_4$.  It is convenient when studying conjugacy classes in $G^0$ to
determine the fixed point set of a non-inner automorphism.
See e.g.,\ \cite[Lem.\ 2.9]{liebeckseitz:2012} when $H$ is cyclic and semisimple (that is, $H$ is of
order coprime to $p$); note that if $H$ is generated by a semisimple
element then $H$ is $G$-cr by Theorem \ref{thm:geometric},
as semisimple conjugacy classes are closed.  
On the other hand, if $H$ is cyclic
and unipotent (that is, $H$ is a $p$-group) then $H$ can be $G$-cr or non-$G$-cr.

We prove the following result, which gives a criterion for $G$-complete
reducibility of $H$.  It is an ingredient in our algorithm.
In case $H$ is cyclic, this is a special case of a recent result
due to Guralnick and Malle, cf.\ Theorem \ref{thm:GM}.

\begin{Thm} [Corollary \ref{cor:outs_on_simple}]
\label{thm:simpleG0}
  Suppose $G^0$ is simple and $H$ acts on $G^0$ by outer automorphisms.
Then $H$ is $G$-completely
reducible if and only if $C_{G^0}(H)$ is reductive. 
\end{Thm}

Our work fits into a study begun in our earlier papers \cite{BMR},
\cite{BMR2}.  It was shown in \cite[Thm.~3.10]{BMR} that if $H$ is a
$G$-cr subgroup of $G$ and $N$ is a normal subgroup of $H$ then $N$ is
also $G$-cr.  In \cite{BMR2} we considered a complementary question: if
$H$ is a subgroup of $G$, $N$ is a normal subgroup of $H$ and $N$ is
$G$-cr then under what hypotheses is $H$ also $G$-cr?  We gave an
example (due to Liebeck) with $H$ of the form $M\times N$, where $M$
and $N$ are both $G$-cr but $H$ is not \cite[Ex.\ 5.3]{BMR2}.  We also showed
this kind of pathological behaviour does not happen when $G$ is
connected and $p$ is good for $G$ \cite[Thm.\ 1.3]{BMR2}.  Here we study the 
above question in the case when $N$ is the normal subgroup $H\cap G^0$ 
of $H$ (see the algorithm in Theorem \ref{thm:algorithm}).

\section{Preliminaries}
\label{sec:prelims}

\subsection{Notation}
Throughout, we work over an algebraically closed field $k$
of characteristic $p \geq0$;
we let $k^*$ denote the multiplicative
group of $k$.
Let $H$ be a linear algebraic group.  
By a subgroup of $H$ we mean a closed subgroup.
We let $Z(H)$ denote the centre of $H$
and $H^0$ the connected component of $H$ that contains $1$.
For $h \in H$, we let $\Int_h$ denote the automorphism of $H$
given by conjugation with $h$.
Frequently, we abbreviate $\Int_h(g)$ by $h \cdot g$. If $S$ is
a subset of $H$ and $K$ is a subgroup of $H$, 
then $C_K(S)$ denotes the centralizer of $S$ in $K$ and
$N_K(S)$ the normalizer of $S$ in $K$. 
Likewise, if $S$ is a group of algebraic automorphisms of 
$H$, then we denote the fixed point subgroup of $S$ in $H$ by 
$C_H(S)$.
If $H$ acts on a set $X$, then 
we also write $C_H(x)$ for the stabilizer
of a point $x \in X$ in $H$.

For the set of cocharacters (one-parameter subgroups) of $H$ we write
$Y(H)$; the elements of $Y(H)$ are the homomorphisms from $k^*$ to $H$.  

The \emph{unipotent radical} of $H$ is denoted $R_u(H)$; it is the
maximal connected normal unipotent subgroup of $H$.  The algebraic
group $H$ is called \emph{reductive} if $R_u(H) = \{1\}$; note that we
do not insist that a reductive group is connected.  In particular, $H$
is reductive if it is simple as an algebraic group.
Here, $H$ is said to be
\emph{simple} if $H$ is connected and all proper normal subgroups of
$H$ are finite.  The algebraic group $H$ is called \emph{linearly
reductive} if all rational representations of $H$ are semisimple.

Throughout the paper $G$ denotes a reductive algebraic group, possibly
non-connected.

\begin{defn}
Let $H \subseteq G$ be a subgroup. We say that
\emph{$H$ acts on $G^0$ by outer automorphisms} 
if for every $1 \neq h \in H$, the automorphism
$\Int_h|_{G^0}$ of $G^0$ is non-inner, i.e.,\ 
is not given by conjugation with an element of $G^0$.
This is equivalent to the condition that
$H$ maps bijectively onto its image under the
natural map
$G \rightarrow \Aut(G^0) \rightarrow \Out(G^0)$.
\end{defn}

\subsection{\texorpdfstring{$G$}{G}-Complete Reducibility}
\label{subsec:noncon}

In \cite[\S6]{BMR}, Serre's original notion of $G$-complete
reducibility is extended to include the case when $G$ is reductive but
not necessarily connected (so that $G^0$ is a connected reductive
group).  The crucial ingredient of this extension is the use
of so-called \emph{Richardson-parabolic subgroups} (\emph{R-parabolic
subgroups}) of $G$. 
We briefly recall the main definitions here; for
more details on this formalism, see \cite[\S6]{BMR}.

For a cocharacter $\lambda
\in Y(G)$, the \emph{R-parabolic subgroup} corresponding to $\lambda$
is defined by $$P_\lambda := \{ g\in G \mid \underset{a\to 0}{\lim}\,
\lambda(a)g\lambda(a)\inverse \textrm{ exists}\}.$$ 
Here, for a morphism of algebraic varieties $\phi:k^*\rightarrow X$,
we say that $\lim_{a \to 0} \phi(a)$ exists provided
that $\phi$ extends to a morphism $\widehat{\phi}:k \rightarrow X$;
in this case we set $\lim_{a \to 0} \phi(a) = \widehat{\phi}(0)$.
Then $P_\lambda$
admits a Levi decomposition $P_\lambda = R_u(P_\lambda) \rtimes
L_\lambda$, where $$L_\lambda = \{ g\in G \mid \underset{a\to
0}{\lim}\, \lambda(a)g\lambda(a)\inverse = g \} = C_G(\lambda(k^*)).$$
We call $L_\lambda$
an \emph{R-Levi subgroup} of $P_\lambda$.  For an R-parabolic subgroup
$P$ of $G$, the different R-Levi subgroups of $P$ correspond in this
way to different choices of $\lambda \in Y(G)$ such that $P =
P_\lambda$; moreover, the R-Levi subgroups of $P$ are all conjugate
under the action of $R_u(P)$.  An R-parabolic subgroup $P$ is a
parabolic subgroup in the sense that $G/P$ is a complete variety; the
converse is true when $G$ is connected, but not in general
(\cite[Rem.\ 5.3]{martin1}).  

\begin{rem}
For a subgroup $H$ of $G$, there is a natural inclusion $Y(H)
\subseteq Y(G)$.  If $\lambda \in Y(H)$, and $H$ is reductive, we can
therefore associate to $\lambda$ an R-parabolic subgroup of $H$ as
well as an R-parabolic subgroup of $G$.  To avoid confusion, we
reserve the notation $P_\lambda$ for R-parabolic subgroups of $G$, and
distinguish the R-parabolic subgroups of $H$ by writing $P_\lambda(H)$
for $\lambda \in Y(H)$.  The notation $L_\lambda(H)$ has the obvious
meaning.  Note that $P_\lambda(H) = P_\lambda \cap H$ and
$L_\lambda(H) = L_\lambda \cap H$ for $\lambda \in Y(H)$.
In particular, $P_\lambda^0=P_\lambda(G^0)$ and 
$L_\lambda^0 = L_\lambda(G^0)$.  If $\lambda\in Y(H)$ then the R-Levi subgroups of $P_\lambda(H)$ are the $R_u(P_\lambda(H))$-conjugates of $L_\lambda(H)$; in particular, any R-Levi subgroup of $P_\lambda(H)$ is of the form $L\cap H$ for some R-Levi subgroup $L$ of $P_\lambda$.
\end{rem}

For later use, we record the following way to construct R-Levi subgroups.

\begin{lem} \label{lem:normalizer}
Let $P$ be an R-parabolic subgroup of $G$,
and let $M$ be a Levi subgroup of $P^0$.
Then $N_P(M)$ is an R-Levi subgroup of $P$.
\end{lem}

\begin{proof}
We may choose $\lambda \in Y(G)$ such that
$P = P_\lambda$, $P^0 = P_\lambda(G^0)$ and $M = L_\lambda(G^0) = 
L_\lambda^0$.
We have the Levi decomposition
$P = R_u(P_\lambda) \rtimes L_\lambda
= R_u(P_\lambda^0) \rtimes L_\lambda$.
Since $L_\lambda \subseteq N_P(L_\lambda^0)$ and
$R_u(P_\lambda^0) \cap N_P(L_\lambda^0) =1$
(as $R_u(P_\lambda^0)$ acts simply transitively on the
set of Levi subgroups of $P_\lambda^0$),
we conclude that $N_P(M) = N_P(L_\lambda^0) = L_\lambda$.
\end{proof}

\begin{defn}
\label{defn:gcr}
Suppose $H$ is a subgroup of $G$.  We say $H$ is \emph{$G$-completely
reducible} ($G$-cr for short) if whenever $H$ is contained in an
R-parabolic subgroup $P$ of $G$, then there exists an R-Levi subgroup
$L$ of $P$ with $H \subseteq L$.
\end{defn}

Since all parabolic subgroups (respectively\ all Levi subgroups of
parabolic subgroups) of a connected reductive group are R-parabolic
subgroups (respectively\ R-Levi subgroups of R-parabolic subgroups),
Definition \ref{defn:gcr} coincides with Serre's original definition
for connected groups \cite{serre2}.

Let $H$ be a subgroup of $G$ and let $G \hookrightarrow \GL_m$
be an embedding of algebraic groups. Let $\tuple{h} \in H^n$ be a
tuple of generators of the associative subalgebra of $\Mat_m$
spanned by $H$ (such a tuple exists for $n$ sufficiently large). Then
$\tuple{h}$ is called a \emph{generic tuple} of $H$, see
\cite[Def.\ 5.4]{GIT}.
We recall the following geometric criterion for $G$-complete reducibility
\cite[Thm.\ 5.8]{GIT}; it provides a link between the theory of $G$-complete reducibility
and the geometric invariant theory of reductive groups. 

\begin{thm} \label{thm:geometric}
Let $H$ be a subgroup of $G$ and let
$\tuple{h}\in H^n$ be a generic tuple for $H$. Then $H$ is
$G$-completely reducible if and only if the orbit $G \cdot \tuple{h}$
under simultaneous conjugation is closed in $G^n$.
In particular, if $H = \langle h \rangle$ is a cyclic subgroup of $G$,
then $H$ is $G$-completely reducible if and only if the conjugacy class
$G\cdot h \subseteq G$ is closed.
\end{thm}

The following result
has been proved with methods from geometric invariant theory
(see \cite[Def.\ 5.17]{GIT}):

\begin{thm} \label{thm:optimal}
Assume that the subgroup $H$ of $G$ is not $G$-completely reducible.
Then there exists an R-parabolic subgroup $P$ of $G$ with the following
two properties:
\begin{itemize}
\item[(i)] $H$ is not contained in any R-Levi subgroup of $P$,
\item[(ii)] $N_G(H) \subseteq P$.
\end{itemize}
\end{thm}

The geometric construction of $P$ in \cite[\S 4]{GIT}
is roughly as follows:
there is a class of so-called
\emph{optimal destabilizing cocharacters} $\Omega \subseteq Y(G)$
such that if $\lambda \in \Omega$ then $P:= P_\lambda$ has properties
(i) and (ii) as in Theorem \ref{thm:optimal}.
We call such an R-parabolic subgroup $P$ of $G$ 
an \emph{optimal destabilizing
R-parabolic subgroup for $H$}. 

\subsection{Criteria for \texorpdfstring{$G$}{G}-complete reducibility}

In this subsection we study criteria for $G$-complete reducibility
in terms of some smaller group.

We first recall the following results about normal subgroups
(cf.\ \cite[Thm.\ 3.4, Cor.\ 3.7(ii)]{BMR2}).

\begin{thm} \label{thm:normal}
Suppose that $N \subseteq H$ are subgroups of $G$ and 
$N$ is normal in $H$.
\begin{itemize}
\item[(i)] If $N$ is also normal in $G$, then $H$ is 
$G$-completely reducible if and only if
$H/N$ is $G/N$-completely reducible.
\item[(ii)] If $H/N$ is linearly reductive, then $H$
is $G$-completely reducible if and only if $N$ is 
$G$-completely reducible.
\end{itemize}
\end{thm}

A homomorphism $\pi: G_1 \rightarrow G_2$ is called \emph{non-degenerate}
provided that $\ker(\pi)^0$ is a torus. 
The next result is contained in \cite[Lem. 2.12 and \S6]{BMR}:

\begin{lem} \label{lem:non-degenerate}
Let $\pi:G_1 \rightarrow G_2$ be a non-degenerate epimorphism
of reductive groups.
Let $H \subseteq G_1$ be a subgroup.
Then $H$ is $G_1$-completely reducible if and only if $\pi(H)$ is 
$G_2$-completely reducible.
\end{lem}

As an immediate consequence, we obtain the following result which 
allows us to focus on the part of $G$ that is effectively acting on $G^0$.
Note that $C_G(G^0)^0\subseteq Z(G^0)^0$ is a torus.

\begin{cor} \label{cor:quotient}
Let $H \subseteq G$ be a subgroup. 
Let $\pi: G \rightarrow G/C_G(G^0)$ be the natural projection.
Then $H$ is $G$-completely reducible if and only if $\pi(H)$ is 
$\pi(G)$-completely reducible.
\end{cor}

The following lemma gives two necessary conditions for a subgroup of $G$ to
be $G$-completely reducible, both of which can be checked in the connected group
$G^0$.

\begin{lem} \label{lem:2_nec_cond}
Let $H$ be a $G$-completely reducible subgroup of $G$. Then the following hold:
\begin{itemize}
\item[(i)] $H \cap G^0$ is $G^0$-completely reducible;
\item[(ii)] $C_{G^0}(H)$ is $G^0$-completely reducible.
\end{itemize}
\end{lem}

\begin{proof}
(i). This is the content of \cite[Lem.\ 6.10 (ii)]{BMR}.
(ii). Since $H$ is $G$-cr, so is its centralizer $C_G(H)$, by
\cite[Cor.\ 3.17 and \S 6]{BMR}. Now the result follows from part (i).
\end{proof}

Under the assumption that assertion (i) or (ii) of Lemma
\ref{lem:2_nec_cond} holds, the next two lemmas allow us
to replace the ambient group $G$ with a potentially smaller subgroup
$M$.

\begin{lem} \label{lem:first_reduction}
Let $H$ be a subgroup of $G$ and suppose that $H \cap G^0$ is 
$G^0$-completely reducible. Then $M = H C_{G^0}(H \cap G^0)$ is
reductive. Moreover, $H$ is $G$-completely reducible if and only if it is
$M$-completely reducible.
\end{lem}

\begin{proof}
First note that $M$ is a well-defined subgroup of $G$, since $H$ normalizes
$C_{G^0}(H \cap G^0)$. As $C_{G^0}(H \cap G^0)$ is a $G^0$-cr
subgroup (Lemma \ref{lem:2_nec_cond}(ii)), it is reductive,
by \cite[Property 4]{serre1}. 
The same is true for $H \cap G^0$ by assumption.
Hence $(H\cap G^0)C_{G^0}(H \cap G^0)$ is 
the product of two reductive groups and thus is reductive.
As this group contains $M^0$ as a normal subgroup, the group $M$
is reductive as well.

Now first suppose that $H$ is not $G$-cr. Let $\lambda \in Y(G) = Y(G^0)$ 
be an optimal destabilizing cocharacter for $H$ in $G$.
Then $P_\lambda(G^0)$ contains the subgroup $H \cap G^0$, which
is $G^0$-cr by assumption. Hence after replacing $\lambda$ with an
$R_u(P_\lambda)$-conjugate, we may assume that $\lambda$ centralizes
$H \cap G^0$. This implies that $\lambda \in Y(M)$, and
$H \subseteq P_\lambda(M) \subseteq P_\lambda$. Since 
$H$ is not contained in an R-Levi subgroup of $P_\lambda$
(cf.\ Theorem \ref{thm:optimal}), it is not
contained in an R-Levi subgroup of $P_\lambda(M)$. We conclude that
$H$ is not $M$-cr.

Conversely, suppose that $H$ is $G$-cr. 
Let $\lambda \in Y(M)$ and suppose that $H \subseteq P_\lambda(M)$.
To show that $H$ is $M$-cr
we need to show that $H$ is contained in an R-Levi subgroup of $P_\lambda(M)$.
Since $H$ is $G$-cr, we can find a cocharacter $\mu \in Y(G)$, such
that $H \subseteq L_\mu \subseteq P_\mu = P_\lambda$.
In particular, $\mu$ centralizes $H \cap G^0$, so that $\mu \in Y(M)$.
Hence $H \subseteq L_\mu(M) \subseteq P_\mu(M) = P_\lambda(M)$,
as required.
\end{proof}

\begin{rem} \label{rem:O3}
In the situation of Lemma \ref{lem:first_reduction}, the 
subgroup $N = H \cap G^0$ of $H$ is normal in $M$.
Thus, by Theorem \ref{thm:normal}(i), we deduce that
$H$ is $G$-completely reducible if and only if
$H/(H\cap G^0)$ is $M/(H \cap G^0)$-completely reducible.
\end{rem}

\begin{lem} \label{lem:second_reduction}
Let $H$ be a subgroup of $G$ and suppose that $C_{G^0}(H)$ is 
$G^0$-completely reducible.
Then $M = H C_{G^0}(C_{G^0}(H))$ is reductive.
Moreover, $H$ is $G$-completely reducible if and only if it is 
$M$-completely reducible. 
\end{lem}

\begin{proof}
We proceed as in the proof of Lemma \ref{lem:first_reduction}:
Again, $M$ is well-defined since $H$ normalizes $C_{G^0}(C_{G^0}(H))$.
Moreover, $M^0 \subseteq (H \cap G^0) C_{G^0}(C_{G^0}(H)) = 
C_{G^0}(C_{G^0}(H)) \subseteq M$, yielding $M^0 = (C_{G^0}(C_{G^0}(H)))^0$.
Since $C_{G^0}(H)$ is $G^0$-cr by assumption, as before we may conclude that
its centralizer is reductive, so that $M$ is reductive.

Suppose that $H$ is not $G$-cr. Let $\lambda \in Y(G) = Y(G^0)$ be an
optimal destabilizing cocharacter for $H$ in $G$.
By Theorem \ref{thm:optimal}(ii), $P_\lambda$ contains
$C_{G^0}(H)$, which is $G^0$-cr.
Thus we may again assume that $\lambda$ centralizes $C_{G^0}(H)$,
so that $\lambda \in Y(M)$.
As before, we conclude that $H$ is not $M$-cr.

Conversely, suppose that $H$ is $G$-cr. Let $\lambda \in Y(M)$ such
that $H \subseteq P_\lambda(M)$. As $\lambda$ evaluates in $M^0$, it
centralizes $C_{G^0}(H)$, so that $C_{G^0}(H)$ is contained 
in $P_\lambda$.
On the other hand, since $H$ is $G$-cr, we may find $\mu \in Y(G)$ 
such that $P_\mu = P_\lambda$ and such that $\mu \in Y(C_{G^0}(H))$. 
But then $P_\mu(C_{G^0}(H)) = C_{G^0}(H)$, which forces
$\mu$ to centralize $C_{G^0}(H)$ \cite[Lem.~2.4]{BMR}.
So $\mu \in Y(M)$, and $H \subseteq L_\mu(M) \subseteq P_\mu(M) = 
P_\lambda(M)$. As before, this shows that $H$ is $M$-cr.
\end{proof}

\begin{rem}
\label{rem:isog}
 Let $H$ be a subgroup of $G$ and let $\pi\colon G\ra G'$ be an isogeny.  Then $\pi(C_{G^0}(H)^0)= C_{G'^0}(\pi(H))^0$ (see the proof of \cite[Lem.\ 3.1]{BMRT}), so $C_{G'^0}(\pi(H))^0$ is reductive if and only if $C_{G^0}(H)^0$ is. 
\end{rem}

We may write the connected reductive group $G^0$ in the form
\begin{equation} \label{eq:components}
G^0 = S G_1 \cdots G_n,
\end{equation}
where $S$ is the radical of $G^0$
and $G_1, \dots, G_n$ are the simple components of the derived
group of $G^0$.
Any subgroup $H$ of $G$ acts via conjugation on the
derived subgroup of $G^0$ and hence permutes the simple components.
We obtain an induced action of $H$ on the set of indices $\{1,\dots,n\}$.
For $1 \leq i \leq n$, we use the shorthand
\begin{equation*}
\widehat G_i = S \prod_{j\neq i} G_j
\end{equation*}
for the product of all factors in $G^0$ above 
with the exception of $G_i$.

Our next lemma allows us to replace $G$ with a collection of reductive groups
whose identity components are simple.

\begin{lem} 
\label{lem:reduce_to_simple}
Let $H$ be a subgroup of $G$.
For $1 \leq i \leq n$,
let $H_i := N_H(G_i)$ and let 
$\pi_i: H_i G^0 \rightarrow H_iG^0 / \widehat G_i$ be
the natural projection.
Let $I \subseteq \{1,\dots,n\}$ be a subset meeting each $H$-orbit.
Then $H$ is $G$-completely reducible if and only if
$\pi_i(H_i)$ is $\pi_i(H_i G^0)$-completely reducible for each
$i \in I$.
\end{lem}

\begin{proof}
First note that, by construction, $H_i$ and $G^0$ both normalize
the group $\widehat G_{i}$. Hence the map $\pi_i$ is
well-defined. Since $H_i G^0$ is reductive, so is its image under 
$\pi_i$.

To prove the forward implication, suppose the assertion fails for some
$i \in I$.
Up to reordering the indices, we may assume
that $\pi_1(H_1)$ is not $\pi_1(H_1 G^0)$-cr and that
$H$ acts transitively on the set $\{1,\dots,r\}$ for some
$r \geq 1$.
Let $Q$ be an optimal destabilizing R-parabolic subgroup of $\pi_1(H_1 G^0)$
for $\pi_1(H_1)$.
To obtain a contradiction, we show that $\pi_1(H_1)$ is contained
in an R-Levi subgroup of $Q$.
Consider the group $Q^0$. This is a parabolic subgroup
of $\pi_1(H_1 G^0)^0 = \pi_1(G^0) = \pi_1 (G_1)$, hence it is of the form
$Q^0 = \pi_1(P_1)$, where $P_1$ is a parabolic subgroup of $G_1$.

Since $P_1$ contains the centre of $G_1$
and $\pi_1(P_1)=Q^0$ is normalized by $\pi_1(H_1)\subseteq Q$,
it follows that $P_1$ is normalized by $H_1$.
Indeed, let $h \in H_1$. Then $h \cdot P_1 \subseteq G_1$.
On the other hand, $\pi_1(h \cdot  P_1) = \pi_1(h) \cdot \pi_1(P_1) = \pi_1(P_1)$.
Since $\ker(\pi_1) = \widehat G_{1}$, this implies
that $h \cdot P_1 \subseteq P_1 \widehat G_1$.
We conclude that $h \cdot P_1 \subseteq P_1 (G_1 \cap \widehat G_1)
= P_1$, where we have used that the last intersection is central in $G_1$.

For $2 \leq j \leq r$, let $h_j \in H$ be an element satisfying
$h_j  \cdot G_1 = G_j$. Let $P_j = h_j \cdot P_1$, which is a parabolic subgroup of $G_j$.
Since we have just verified that $H_1$ normalizes $P_1$,
the definition of $P_j$ does not depend on the choice of $h_j$ that
transports $G_1$ to $G_j$.

We now consider the parabolic subgroup
$P = S P_1 \cdots P_n$ of $G^0$, where we take $P_j = G_j$ for
$j > r$.
By construction, $P$ is normalized by $H$. Indeed, any $h \in H$
fixes $S$ under conjugation, and permutes the groups $G_1,\dots,G_n$.
If $h$ maps $G_i$ to $G_j$ with $i,j \in \{1,\dots,r\}$,
then $(hh_i) \cdot G_1 = G_j$, and hence $h \cdot P_i = (hh_i) \cdot P_1 = P_j$.
So $h$ also permutes the groups $P_1,\dots,P_r$, and thus
normalizes $P$.

The group $N_G(P)$ is thus an R-parabolic subgroup of $G$
containing $H$ with $N_G(P)^0 = P$
(see \cite[Prop.\ 6.1]{BMR}).
Since $H$ is $G$-cr, it is contained in an R-Levi subgroup $L$ of $N_G(P)$,
hence it normalizes the Levi subgroup $L^0$ of $P$.
We may write $L^0 = S L_1 \cdots L_n$ for certain Levi subgroups $L_j$ of $P_j$.
Then $H_1$ normalizes $L_1$, since $L_1 = L^0 \cap G_1$.
This forces $\pi_1(H_1)$ to normalize a Levi subgroup of
$Q^0 = \pi_1(P_1)$. 
By Lemma \ref{lem:normalizer}, $\pi_1(H_1)$ is 
contained in an R-Levi subgroup of $Q$, yielding a contradiction.

To prove the reverse implication, we again assume after reordering the
indices that $1 \in I$ and that $H$ permutes the set $\{1,\dots,r\}$
transitively for some $r \geq 1$. Assume that $H$ is not $G$-cr,
and that $Q \subseteq G$ is an optimal destabilizing R-parabolic 
subgroup of $G$ containing $H$. Again we want to deduce that $H$
is contained in an R-Levi subgroup of $Q$, contradicting our assumption.

Write $Q^0 = S P_1 \cdots P_n$, where the $P_i$ are parabolic subgroups
of $G_i$. Since $H$ normalizes $Q^0$, $H_1$ normalizes
$P_1 = Q^0 \cap G_1$. This means that $\pi_1(H_1)$ is contained
in $N_{\pi_1(H_1G^0)}(\pi_1(P_1))$, and the latter is an R-parabolic
subgroup of $\pi_1(H_1G^0)$. Since $1 \in I$, the subgroup $\pi_1(H_1)$
is contained in an R-Levi subgroup, say $M$, of this normalizer. Thus
$\pi_1(H_1)$ normalizes 
the Levi subgroup $M^0$ of $\pi_1(P_1)$, which is hence of the form
$\pi_1(L_1)$ for some Levi subgroup $L_1$ of $P_1$.

Since $L_1$ contains the centre of $G_1$, as in the proof of the forward
implication (where we have proved that $H_1$ normalizes $P_1$), we may
conclude that $H_1$ normalizes $L_1$.
Choosing again elements $h_j \in H$ with $h_j \cdot G_1 = G_j$ for 
$2 \leq j \leq r$, we obtain well-defined Levi subgroups
$L_j := h_j \cdot L_1$ of $h_j \cdot P_1 = P_j$, where the latter equality
follows from $P_j = Q^0 \cap G_j$.
Proceeding similarly for the other $H$-orbits on $\{1,\dots,n\}$ 
(each of which contains an element of $I$ by assumption),
we construct an $H$-stable Levi subgroup
$L = SL_1\cdots L_n$ of $Q^0$. As before, 
by Lemma \ref{lem:normalizer}, $H$ is contained in an R-Levi subgroup of $Q$,
which gives the desired contradiction. 
This finishes the proof.
\end{proof}

\section{Cyclic subgroups}

According to Theorem \ref{thm:geometric}, a cyclic subgroup of
$G$ is $G$-completely reducible if and only if the conjugacy class
of a generator is closed. With this characterization we reformulate
a recent result of Guralnick and Malle, see \cite[Thm.\ 2.3]{GM:2013}:

\begin{thm} 
\label{thm:GM}
Let $H$ be a cyclic subgroup of $G$. Then
$H$ is $G$-completely reducible if and only if
$C_{G^0}(H)$ is reductive.
\end{thm}

By Lemma \ref{lem:2_nec_cond}(ii), we deduce the following result.

\begin{cor} \label{cor:cent_of_cyclic}
Let $H$ be a cyclic subgroup of $G$.
Then $C_{G^0}(H)$ is $G^0$-completely reducible if and only if it is reductive.
\end{cor}

Combining some of our previous reductions,
we obtain the following weaker version of 
 Theorem \ref{thm:GM}.
This is of independent interest, as our arguments allow us to avoid 
the case-by-case considerations that are needed for the
proof of Theorem \ref{thm:GM}.

\begin{thm} \label{thm:cyclic}
Let $H \subseteq G$ be a cyclic subgroup 
of prime order.
Then $H$ is $G$-completely reducible if and only if $C_{G^0}(H)$ is 
reductive.
\end{thm}

\begin{proof}
The forward implication is clear, by Lemma \ref{lem:2_nec_cond}(ii).
Conversely, assume that $C_{G^0}(H)$ is reductive.
Since linearly reductive subgroups are $G$-cr
(see \cite[Lem.\ 2.6]{BMR}), we may
assume that $k$ has positive characteristic $p$ that
coincides with the order of $H$.

We first show that $C_{G^0}(H)$ is $G^0$-cr.
Suppose this fails, and let $P \subseteq G^0$ be an 
optimal destabilizing parabolic subgroup for $C_{G^0}(H)$ in $G^0$.
Then $H$ normalizes $P$, by Theorem \ref{thm:optimal}(ii).
Let $U$ be the unipotent radical of $P$.
Then $Z(U)$ has positive dimension and is normalized by $H$ and $P$.
Up to passing to a characteristic subgroup
(the subgroup of elements of order dividing $p$),
we may assume that $Z(U)$ has exponent $p$. 
Thus $Z(U)$ has the structure of an 
$\FF_p$-vector space of infinite dimension
with an $\FF_p$-linear $H$-action.
As $H$ is generated by an element $h \in H$ of order $p$, there
must be infinitely many fixed points of $H$ on $Z(U)$.
Indeed, on any $H$-stable finite dimensional 
subspace $W$ of $Z(U)$ the automorphism
induced by $h$ may be brought into Jordan normal form with
block sizes bounded by $p$ (the Jordan normal form exists as $h$ has only 
eigenvalue $1 \in \FF_p$).
As each block contributes at least
$p-1$ fixed points, $H$ has at least $(p-1)[\dim W/p]$ fixed
points on $W$, and we can make $\dim_{{\mathbb F}_p} W$ arbitrarily large. 
Taking the identity component of the $H$-fixed points
on $Z(U)$ hence yields a 
non-trivial, connected, normal, unipotent subgroup of $C_{G^0}(H)$,
contradicting the reductivity assumption. We thus
conclude that $C_{G^0}(H)$ is $G^0$-cr.

By Lemma \ref{lem:second_reduction},
it therefore suffices to show that $H$ is $M$-cr,
where $M = H C_{G^0}(C_{G^0}(H))$.
Let $M_1,\dots,M_r$ be the simple components of $M^0$.
By definition of $M$, $C_{M_i}(H) \subseteq M_i \cap Z(M^0)$ is
finite for each $i$.
This forces $H_i = N_{H}(M_i) = 1$ for each $i$.
Indeed, since $H$ has prime order,
$H_i \neq 1$ would yield that $H$ normalizes $M_i$,
but due to a result of Steinberg (cf.\ \cite[Thm.\ 10.13]{steinberg}),
no non-trivial cyclic group can act on a simple group via
algebraic automorphisms with only finitely many fixed points.
By Lemma \ref{lem:reduce_to_simple},
we conclude that $H$ is $M$-cr, as required.
\end{proof}

\section{Outer automorphisms for \texorpdfstring{$D_4$}{D4}}
\label{sec:D4}

In this section, let $D_4$ denote an adjoint simple group of type $D_4$.
Amongst the simple groups
$D_4$ has the largest outer automorphism group,
in that $\Out(D_4)\cong S_3$,
the symmetric group on $3$ letters.
We may identify $\Out(D_4)$ with the set of graph automorphisms
in $\Aut(D_4)$
induced by the symmetries of the Dynkin diagram.
However, there are other subgroups isomorphic to $S_3$ in $\Aut(D_4)$
that act via outer automorphisms.
As this is the only situation where outer automorphisms of
a simple group arise that is not covered by
Theorem \ref{thm:cyclic}, we treat this case 
separately in this section.

Let $T$ be a maximal torus of $D_4$ with associated root system
$\Phi$. Let $\Delta=\{\alpha,\beta,\gamma,\delta\}$ be a set of simple
roots for $\Phi$, where $\delta$ is the unique simple root that is
non-orthogonal to every other simple root.
Let $\lambda = \omega_\delta^\vee \in Y(T)$ be the fundamental dominant
coweight determined by 
$\langle \alpha, \lambda \rangle = 
\langle \beta, \lambda \rangle =
\langle \gamma, \lambda \rangle = 0$, $\langle \delta, \lambda \rangle = 1$.
For $\epsilon \in \Phi$ we denote by $u_\epsilon:
\mathbb{G}_a \rightarrow U_\epsilon$ 
a fixed root homomorphism onto the corresponding root subgroup of $G$.

Let $\sigma \in \Aut(D_4)$ be the triality graph automorphism determined by
requiring that for all $c\in k$,
\begin{align*}
\sigma(u_\alpha(c))& = u_\beta(c),& 
\sigma(u_\beta(c))& = u_\gamma(c),& 
\sigma(u_\gamma(c))& = u_\alpha(c),&
\sigma(u_\delta(c))& = u_\delta(c).
\end{align*}
Then $C_{D_4}(\sigma)$ is a simple group of type $G_2$. In fact,
$\tilde{T} = C_T(\sigma)$ is a maximal torus of $C_{D_4}(\sigma)$,
and $\tilde{\alpha} = \alpha|_{\tilde{T}} = \beta|_{\tilde{T}}
= \gamma|_{\tilde{T}}$ and $\tilde{\beta} = \delta|_{\tilde{T}}$ 
form a pair of simple roots with respect to $\tilde{T}$,
with corresponding root groups given by
$u_{\tilde{\alpha}}(c) = u_\alpha(c)u_\beta(c)u_\gamma(c)$,
$u_{\tilde{\beta}}(c) = u_\delta(c)$.
Since $\lambda$ evaluates in $\tilde{T}$, we may regard it as
an element of $Y(\tilde{T})$; we denote this element by
$\tilde{\lambda}$.
We have $\langle \tilde{\alpha}, \tilde{\lambda} \rangle = 0$,
$\langle \tilde{\beta}, \tilde{\lambda} \rangle = 1$.

We begin with a detailed description of triality in the 
particular case where the 
ground field has characteristic three, using the 
results
of \cite{CKT} and \cite{elduque}.

\begin{prop} \label{prop:triality}
Assume that $p=3$.
In $\Aut(D_4)$ there are exactly two conjugacy classes of cyclic
groups of order three generated by outer automorphisms.
Let $\langle \sigma_1 \rangle$, $\langle \sigma_2 \rangle$
be representatives of the respective classes,
and let $M_i = C_{D_4}(\sigma_i)$ ($i=1,2$).
Then we may choose
the labelling such that the following holds:
\begin{itemize}
\item[(i)] $M_1$ is a simple group of type $G_2$;
moreover $\Aut(D_4)\cdot \sigma_1$, the orbit of $\sigma_1$ under conjugation,
is closed in $\Aut(D_4)$.
\item[(ii)] 
$M_2$ is an $8$-dimensional group
with $5$-dimensional unipotent radical and corresponding reductive
quotient isomorphic to $\SL_2$; the orbit $\Aut(D_4)\cdot \sigma_2$ is not
closed and contains $\sigma_1$ in its closure.

\item[(iii)] We may take $\sigma_1 = \sigma$.

\item[(iv)]
Let $u =  u_{\alpha + \beta + \gamma + 2\delta}(1)$.
Then we may take $\sigma_2= \sigma u$.

\item[(v)] With the choices in (iii) and (iv), we have
$M_2 = C_{M_1}(u) = \langle U_{\tilde{\alpha}}, U_{-\tilde{\alpha}}\rangle
\ltimes R_u(P_{\tilde{\lambda}}) \subseteq P_{\tilde{\lambda}}$, where $P_{\tilde{\lambda}}$ denotes $P_\lambda(M_1)$.
\end{itemize}
\end{prop}

\begin{proof}
By \cite[Cor.\ 6.5, Thm.\ 9.1]{CKT}, there are precisely two
conjugacy classes of cyclic groups of order three generated by
outer automorphisms, which are denoted by type I and type II,
respectively.
They are distinguished by the structure of the corresponding
fixed point groups, where type I yields a group of type $G_2$,
whereas type II in characteristic 3 gives a group with the
structure described in (ii) (see \cite[\S9]{CKT} together with 
\cite[Thm.\ 7]{elduque}). This implies the first statements of
(i) and (ii), as well as (iii).

Working in the algebraic group $\Aut(D_4)$, 
using $\sigma(\lambda) = \lambda$
and $\langle \alpha + \beta + \gamma + 2\delta , \lambda \rangle
= 2 > 0$
we compute that
\begin{equation*}
\lim_{a \rightarrow 0} \lambda(a) \sigma u \lambda(a)^{-1}
= \sigma \lim_{a \rightarrow 0} \lambda(a) u \lambda(a)^{-1} =
\sigma,
\end{equation*}
which proves that $\sigma$ is contained in the closure of 
the orbit through $\sigma u$.
By \cite[Thm.\ 3.3]{GIT}, to show that $\sigma$ and 
$\sigma u$ are not conjugate it is enough to show that they are
not $D_4$-conjugate. 
This non-conjugacy follows from 
\cite[\S I, Prop.\ 3.2]{spalt}, as $\sigma$ and $\sigma u$
are given in \emph{loc.~cit.}\ as class representatives for distinct unipotent
classes in $\sigma D_4$. 
Moreover, by \cite[\S II, Lem.\ 1.15]{spalt}, the orbit
through $\sigma$ is closed. This proves (iv) and 
the remaining assertions of (i) and (ii).

To prove (v), we first note that $R_u(P_{\tilde{\lambda}})$ 
consists of the root groups for the roots
$\tilde{\beta}, \tilde{\alpha}+\tilde{\beta}$,
$2\tilde{\alpha} + \tilde{\beta}$,
$3\tilde{\alpha} + \tilde{\beta}$ and
$3\tilde{\alpha} + 2\tilde{\beta}$.
In particular, the semi-direct product 
$\langle U_{\tilde{\alpha}}, U_{-\tilde{\alpha}} \rangle
\ltimes R_u(P_{\tilde{\lambda}})$ has dimension 8 and is
contained in $C_{M_1}(u) = C_{M_1}(u_{3\tilde{\alpha}
+2\tilde{\beta}}(1))$.
As $\tilde{\lambda}$ centralizes $\pm\tilde{\alpha}$,
the semi-direct product is also contained in $P_{\tilde{\lambda}}$.
Since clearly $C_{M_1}(u) \subseteq M_2$, the assertion (v) follows
by comparing dimensions. This finishes the proof.
\end{proof}

We can now characterize $G$-complete reducibility in
the case where $G^0 = D_4$ and $H$ maps isomorphically onto the full
group of outer automorphisms of $D_4$.
The following result is the analogue of Theorem \ref{thm:cyclic}.

\begin{thm} \label{thm:D4}
Let $H$ be a subgroup of $G$. Assume that
$G^0 = D_4$ and $H \cong S_3$ acts by outer automorphisms on
$G^0$.
Then $H$ is $G$-completely reducible if and only if $C_{G^0}(H)$ is 
reductive.
\end{thm}

\begin{proof}
The forward implication is clear by Lemma \ref{lem:2_nec_cond}. 
Conversely, assume that $H$ is not $G$-cr.  Let $h \in H$ be an element of order $3$, so that
$K = \langle h \rangle$ is a normal subgroup of index $2$ in
$H$.  By the assumption on $H$,
the map $\pi:G \rightarrow \Aut(D_4), g \mapsto \Int(g)|_{G^0}$ is
surjective. Since $\ker(\pi)=C_G(G^0)$, $\pi$ is an
isogeny.
Hence $\pi(H)$ is not $\Aut(D_4)$-cr, by
Lemma \ref{lem:non-degenerate}.  It now follows from Remark~\ref{rem:isog} that we can take $G$ to be $\Aut(D_4)$.

First assume that $p=3$. 
Let $M_1 = C_{D_4}(\sigma)$
and $M_2 = C_{D_4}(\sigma u)$ with notation as in Proposition
\ref{prop:triality}.  Then $K$
is a normal subgroup of order $3$ and index $2$ in $H$.
Since $p=3$ is coprime to $2$, we have by Theorem \ref{thm:normal}(ii) that 
$K$ is not $\Aut(D_4)$-cr.
This implies (by Theorem \ref{thm:geometric})
that the orbit $\Aut(D_4)\cdot h$ is not
closed, whence by Proposition \ref{prop:triality} there exists
$g \in G$ with $g h g^{-1} = \sigma u$.
Replacing $H$ with $gHg^{-1}$, we may assume
that $h = \sigma u$.
Let $s \in H$ be an element of order $2$ such that $h$ and $s$
generate $H$.
Let $\tau \in \Aut(D_4)$ be the graph automorphism of order $2$
determined by $s$, i.e.,\ the graph automorphism that
induces the same element as $s$ in $\Out(D_4)$.
Let $t = \tilde{\beta}^\vee(-1) \in \tilde{T} \subseteq M_1$.
Since $\tau$ and $\sigma$ fix $M_1$, both elements commute
with $t$ and $u$.
Moreover, by construction $t u t = u^{-1}$.
This implies that $\tau t$ has order $2$ and
$(\tau t)(\sigma u) (\tau t) = \tau \sigma \tau u^{-1} =
\sigma^{-1} u^{-1} = (\sigma u)^{-1}$.
As $\tau t$ induces the same element as $s$ in $\Out(D_4)$,
we can find $x \in D_4$ with $s = \tau t x$.
We conclude that both pairs of elements 
$h = \sigma u, s = \tau t x$ as well as
$\sigma u, \tau t$ generate a group isomorphic to $S_3$.
In particular,
$x (\sigma u) x^{-1} = (\tau t)(\tau t x)(\sigma u)(\tau t x)^{-1}(\tau t)
= (\tau t) (\sigma u)^{-1} (\tau t) = \sigma u$.
Thus $x \in M_2 \subseteq M_1$ (cf. Proposition \ref{prop:triality}(v)),
so that $s = \tau t x$ normalizes $M_1$.
As $M_1$ is simple of type $G_2$, it has no outer automorphisms.
Therefore we may find $s' \in M_1$ with
$\Int(s)|_{M_1} = \Int(s')|_{M_1}$.
Since $M_1$ is adjoint, $s'$ is of order $2$.
Now
\begin{equation}
\label{eq:a}
C_{G^0}(H) = C_{M_2}(s').
\end{equation}
Since $s = \tau t x$ normalizes $M_2$, $s$ normalizes $N_{M_1}(M_2)= P_{\tilde{\lambda}}$ (see Proposition \ref{prop:triality}(v)).
Hence $s' \in P_{\tilde{\lambda}}$
and $C_{M_2}(s') \subseteq P_{\tilde{\lambda}}$.
Up to conjugation in $P_{\tilde{\lambda}}$ we may thus
assume $s' \in \tilde{T}$. As $\tilde{T}$ is generated
by the images of $\tilde{\alpha}^\vee$ and $\tilde{\beta}^\vee$,
this reduces the possibilities to
$s' \in \{\tilde{\alpha}^\vee(-1),
\tilde{\beta}^\vee(-1),
\tilde{\alpha}^\vee(-1)\tilde{\beta}^\vee(-1)\}$.
But then $s'$ centralizes
$U_{3\tilde{\alpha}+2\tilde{\beta}}$,
or $U_{\tilde{\beta}}$, or $U_{\tilde{\alpha}+\tilde{\beta}}$ 
respectively.
We deduce that
$C_{M_2}(s') \cap R_u(P_{\tilde{\lambda}})$ is positive-dimensional
in each case. Thus $C_{M_2}(s')$ is not reductive, as required.

Now let $p \neq 3$.
Then the subgroup $K$ of $H$ of order $3$ is linearly reductive,
in particular it is $G$-cr and $C_{G^0}(K)$ is reductive. 
Moreover, the group $C_{G^0}(K)$ is connected
being the fixed point group under a triality automorphism
(cf.\ \cite[\S9]{CKT}).
Let $M = H C_{G^0}(K)$.
By \cite[Thm.\ 3.1(b)(ii)]{BMR2} applied to $K \subseteq H \subseteq M$,
we deduce that $H$ is not $M$-cr.
Since $K$ is normal in $M$, by 
Theorem \ref{thm:normal}(i),
$H/K$ is not $M/K$-cr.
But $H/K$ is cyclic of order $2$, so we may apply
Theorem \ref{thm:cyclic}
to conclude that $C_{(M/K)^0}(H/K)$ is not reductive. 
By construction,
$(M/K)^0 \cong C_{G^0}(K)$ and 
\begin{equation}
\label{eq:b}
C_{(M/K)^0}(H/K) \cong C_{G^0}(H).
\end{equation}
This finishes the proof.
\end{proof}

Having settled the case of $D_4$, we can combine
Theorems \ref{thm:cyclic} and \ref{thm:D4}
to characterize $G$-complete reducibility 
in case $G^0$ is simple and the
subgroup $H$ acts by outer automorphisms.

\begin{cor} \label{cor:outs_on_simple}
Let $H \subseteq G$ be a subgroup acting on $G^0$
by outer automorphisms.
Assume that $G^0$ is simple.
Then $H$ is $G$-completely reducible if and only if $C_{G^0}(H)$ is 
reductive.
\end{cor}

\begin{proof}
We may assume that $G^0$ is adjoint (cf.~the first paragraph of the proof of Theorem~\ref{thm:D4}).
Since $H$ acts via outer automorphisms,
we may identify it as an abstract group 
with a subgroup of $\Out(G^0)$, the finite group of 
outer automorphisms of $G^0$.
As $G^0$ is simple, $\Out(G^0)$ is either simple of prime order
or $G^0$ is of type $D_4$ and $\Out(G^0) \cong S_3$.
The result now follows
from Theorems \ref{thm:cyclic} and \ref{thm:D4}.
\end{proof}

\begin{rem}
In the situation of Corollary \ref{cor:outs_on_simple},
we always have $\dim C_{G^0}(H) > 0$. This follows
again from the theorem of Steinberg (\cite[Thm.\ 10.13]{steinberg}) in 
case $H$ is cyclic. The general case follows from the
identities for $C_{G^0}(H)$ in \eqref{eq:a} and \eqref{eq:b}.
\end{rem}

\section{The Algorithm}

We return to the general situation where
$H \subseteq G$ is a subgroup of a 
possibly non-connected reductive group.
In this section, we are going to establish
an algorithm
that reduces the question of whether $H$ is $G$-cr
to the question of whether certain subgroups of certain
\emph{connected} reductive groups are $G$-completely reducible.

We start with a proposition that recasts some of
our earlier results as operations for a potential
algorithm. We say that the pair $(H,G)$ is \emph{completely
reducible} provided that $H$ is $G$-completely reducible.

\begin{prop} \label{prop:operations}
Let $(H,G)$ be a pair consisting of a reductive group $G$ and 
a subgroup $H \subseteq G$. Then each of the following operations
replaces $(H,G)$ with pairs of the same form (i.e.,\ consisting
of a group and a reductive group containing it as a subgroup):
\begin{itemize}
\item[(O1)]
Let $\pi: G \rightarrow G/C_G(G^0)$ be the canonical projection.
Replace $(H,G)$ with $(\pi(H),\pi(G))$.
\item[(O2)]
Let $G^0=SG_1\cdots G_n$ be the decomposition as in 
\eqref{eq:components}, and let $H_i$, $\pi_i$ for $1 \leq i \leq n$
be defined as in Lemma \ref{lem:reduce_to_simple}.
Replace $(H,G)$ with the pairs 
$$(\pi_1(H_1),\pi_1(H_1G^0)),
\dots ,(\pi_n(H_n),\pi_n(H_nG^0)).$$
\item[(O3)]
If $H \cap G^0$ is $G^0$-completely reducible, replace $(H,G)$ with
$$(H/(H\cap G^0), HC_{G^0}(H\cap G^0)/ (H\cap G^0)).$$
\end{itemize}
Moreover, $H$ is $G$-completely reducible if and only if
each of the pairs obtained through one of these operations
is completely reducible.
\end{prop}

\begin{proof}
The results follow from
Corollary \ref{cor:quotient},
Lemma \ref{lem:reduce_to_simple}
and Remark \ref{rem:O3}.
\end{proof}

\begin{rem} \label{rem:step2_I}
In the situation of (O2), suppose we are given a set $I$ as in
Lemma \ref{lem:reduce_to_simple}. 
Then it is enough to replace
$(H,G)$ with the pairs $(\pi_i(H_i),\pi_i(H_iG^0))_{i\in I}$ in
(O2).
\end{rem}

We are now in a position to give an algorithm
that determines whether $H$ is $G$-completely reducible.

\begin{thm} \label{thm:algorithm}
Let $H$ be a subgroup of $G$. The following algorithm,
starting with the pair $(H,G)$,
reduces the question of whether $H$ is $G$-completely reducible
in a finite number of steps to questions of complete
reducibility in connected groups:

\begin{Alg}
Input: a pair $(H',G')$, where $H'$ is a subgroup of a reductive
group $G'$.
\begin{itemize}
\item[Step 1] 
If $G'^0$ is not simple, or $C_{G'}(G'^0)\neq 1$,
apply (O2) and then (O1) to each of the newly
obtained pairs. Restart instances of the algorithm for each of the new
pairs. Then $H'$ is $G'$-cr if and only if each of these pairs turns
out to be completely reducible.

\item[Step 2]
Identify $H'/(H' \cap G^0)$ with a subgroup of $\Out(G'^0)$.
Let $n \in \{1,2,3,6\}$ be the order of $H'/(H' \cap G'^0)$.
If $p$ does not divide $n$, $H'$ is $G'$-cr
if and only if $H' \cap G'^0$ is $G'^0$-cr,
and the algorithm stops.

\item[Step 3]
If $H' \cap G'^0$ is not $G'^0$-cr,
the algorithm stops with the conclusion that $H'$ is not $G'$-cr.

\item[Step 4]
If $H' \cap G'^0 = 1$, 
$H'$ is $G'$-cr if and only if
$C_{G'^0}(H')$ is reductive. The algorithm stops.

\item[Step 5]
If $1 \neq H'/(H' \cap G'^0) \not\cong S_3$,
let $M = H' C_{G'^0}(H' \cap G'^0)/(H' \cap G'^0)$.
Then $H'$ is $G'$-cr
if and only if
$C_{M^0}(H'/(H'\cap G'^0))$ is $M^0$-cr. The algorithm stops.

\item[Step 6]
If $H'/(H' \cap G'^0) \cong S_3$,
apply (O3) and
restart the algorithm with the new pair.
\end{itemize}
\end{Alg}
\end{thm}

\begin{proof}
Step 1 is covered by Proposition \ref{prop:operations}.
Moreover, each of the new pairs $(H'',G'')$
produced in this step
satisfies $G''^0$ simple and $C_{G''}(G''^0) = 1$, so 
that the algorithm moves on to Step 2 after a possible application of Step 1.

From Step 2 on, we may assume that $G'^0$ is simple and $C_{G'}(G'^0) = 1$.
This allows us to identify $H'/(H' \cap G'^0)$ with
a subgroup of $\Out(G'^0)$, and yields the constraints on its order.
If $p$ does not divide $n$, the quotient $H'/(H' \cap G'^0)$ is
linearly reductive.
By Theorem \ref{thm:normal}(ii), $H'$ is indeed $G'$-cr if
and only if $H'\cap G'^0$ is $G'^0$-cr.

From Step 3 on, we may assume in addition that $p \in \{2,3\}$
and that $H$ is not contained in $G^0$.
The conclusion of Step 3 is correct by Lemma \ref{lem:2_nec_cond}(i).

Step 4 is an application of Corollary \ref{cor:outs_on_simple}.

Since we have passed Step 3, we may assume that $H' \cap G'^0$
is $G'^0$-cr.
Under the condition of Step 5, $H'/(H'\cap G'^0)$ is cyclic of prime
order. The conclusion of Step 5 thus follows from
Remark \ref{rem:O3} and Theorem \ref{thm:cyclic}.

Finally, Step 6 is again covered by Proposition \ref{prop:operations}.
Moreover, this step is only applicable for $G'^0$ simple of type $D_4$.
As we may assume $H' \cap G'^0 \neq 1$ and $Z(G'^0)=1$, the group 
$M = H'C_{G'^0}(H' \cap G'^0)/(H' \cap G'^0)$ featuring in (O3)
satisfies $\dim M < \dim G'$.

It remains to show that the algorithm terminates. 
Step 1 may restart finitely many instances of the algorithm. In each instance
the algorithm terminates in Step 2 -- Step 5 if Step 6 is not reached.
If Step 6 is applicable, it replaces $G'$---which is simple of type $D_4$---with a group of smaller
dimension. This implies that after Step 1 is applied again, Step 6 cannot 
be reached a second time, and the algorithm terminates.
\end{proof}

\begin{rem}
(i). It follows from the proof of Theorem \ref{thm:algorithm}
that Step 1, the only step that replaces a pair with several new pairs,
need only be done at most twice along a path through the algorithm. Also,
Step 6 only occurs at most once.

(ii). There are some situations where shortcuts may be applied to reduce
to a connected group. 
First of all, if $H^0$ is not reductive, then $H$ cannot be $G$-cr.
On the other hand, if $H$ is cyclic, then we may apply 
Theorem \ref{thm:GM} to deduce that $H$ is $G$-cr if and only if 
$C_{G^0}(H)$ is $G^0$-cr. Finally, if $H/(H \cap G^0)$
is linearly reductive, we can apply Theorem \ref{thm:normal}(ii) to deduce
that $H$ is $G$-cr if and only if $H \cap G^0$ is $G^0$-cr.
However, the proposed algorithm gives a systematic approach that deals with
all possible cases.

(iii). If $p=0$, then a subgroup $H$ is $G$-cr if and only if it
is reductive (\cite[Prop.\ 4.2]{serre2}).
Of course, $H$ is reductive if and only if
$H^0$ is reductive, which in turn is equivalent to $H^0$ being
$G^0$-completely reducible.
\end{rem}

\section{Examples}
\label{sec:ex}

We conclude 
with some examples of the algorithm outlined in
Theorem \ref{thm:algorithm}.

\begin{exmp}
Let $p=3$, $G = \Aut(D_4)$. Let $\sigma$ be the triality
graph automorphism as in Section \ref{sec:D4}.
Let $H = \langle \sigma \rangle K$, where
$K = C_{D_4}(\sigma)$ is the fixed point subgroup of type $G_2$.
We follow through the algorithm to deduce that $H$ is $G$-cr:

Step 1 is not applicable, as $G^0 = D_4$ is simple and
$C_G(G^0) = 1$.
In Step 2 we obtain $n=3=p$ as the order of
$\langle \sigma \rangle \cong H/(H \cap G^0)$.
Now $H \cap G^0 = K$ is $G$-cr
(see Corollary \ref{cor:cent_of_cyclic}),
hence Steps 3 and 4 are not applicable.
Step 5 applies and leads us to consider the group 
$M = HC_{D_4}(K)/K \cong \langle \sigma \rangle C_{D_4}(K)$.
As $K$ is adjoint, we obtain $C_{M^0}(\sigma) = 1$ and thus
this group is clearly $M^0$-cr.
The algorithm stops with the conclusion that $H$ is $G$-cr.

Here we have two commuting $G$-cr subgroups $\langle \sigma \rangle$ and $K$ of $G$ and their product is also $G$-cr.  This is not always the case: see \cite[Ex.~5.1]{BMR2}.
\end{exmp}

\begin{exmp}
Let $\Gamma$ be a finite group acting
transitively on a finite set $I$. Let $i_0 \in I$.
Let $\rho:\Gamma \rightarrow M$ be a homomorphism 
to a simple group $M$ such that $\rho(C_\Gamma(i_0))$ is 
not $M$-cr.
We set
\begin{equation*}
G = \Gamma \ltimes \prod_{i \in I} M,
\end{equation*}
where $\Gamma$ acts on the product by permuting the indices.
Clearly, $G^0 = \prod_{i} M$.
Let $d:M \rightarrow G^0$ be the diagonal embedding.
As $\Gamma$ commutes with the image of $d$, we may 
form the subgroup
\begin{equation*}
H = \{ \gamma d(\rho(\gamma)) \mid \gamma \in \Gamma \},
\end{equation*}
a finite subgroup of $G$. 
We claim that $H$ is not $G$-cr, and use our algorithm to prove it.

To apply Step 1, we compute that for $i \in I$, $N_G(G_i) 
= C_\Gamma(i) \ltimes G^0$. 
In particular, $H_i = \{\gamma d(\rho(\gamma))
\mid \gamma \in C_\Gamma(i) \}$.
We obtain
\begin{equation*}
\pi_i(H_i G^0) =H_iG^0 / \prod_{j \neq i} G_j \cong C_\Gamma(i) \times M
=: G'_i,
\end{equation*}
and correspondingly
\begin{equation*}
\pi_i(H_i) \cong \{\gamma \rho(\gamma) \mid \gamma \in C_\Gamma(i) \} 
=: H'_i.
\end{equation*}
Since the action of $\Gamma$ on $I$ is transitive, we may by
Remark \ref{rem:step2_I} replace $(H,G)$ with $(H',G')$,
where $G' = G'_{i_0}, H' = H'_{i_0}$.

Now $C_{G'}(G'^0) = C_\Gamma(i_0)$, so Step 1 is again applicable and replaces
$(H',G')$ with the pair $(H'',M)$,
where
\begin{equation*}
H'' = \rho(C_\Gamma(i_0)).
\end{equation*}
By assumption, $H''$ is not $M$-cr and hence $H$ is not $G$-cr, by Step 2.

As a concrete realisation of this example, take
$\Gamma = \PGL_2(q)$ for
$q$ a sufficiently large power of $p$, $I = \PGL_2(q)/B(q)$ where $B$ is a 
Borel subgroup of $\PGL_2$, and consider $I$ as a transitive $\Gamma$-set
by left translation.
Let $M = \PGL_2$, and let $\rho: \Gamma \hookrightarrow M$
the canonical embedding.
If we take $i_0 = B(q)$, then $C_\Gamma(i_0) = B(q)$ is not $M$-cr,
for $q$ large enough, as $B$ is not $M$-cr.
In this example, we have $H \cap G^0 = 1$ and
$C_{G^0}(H) = 1$ (as $C_M(B(q))=1$ for $q$ sufficiently large).
In particular, this example shows that $H$ may fail to be
$G$-cr even if $C_{G^0}(H)$ and $H\cap G^0$ both are $G^0$-cr.
\end{exmp}


\bigskip {\bf Acknowledgements}: 
The authors acknowledge the financial support of 
the DFG-priority programme SPP1388 ``Representation Theory''
and Marsden Grants UOC0501, UOC1009 and UOA1021.
Part of the research for this paper was carried out while the
authors were staying at the Mathematical Research Institute
Oberwolfach supported by the ``Research in Pairs'' programme.
The second author acknowledges additional support from
ERC Advanced Grant 291512.

\end{document}